\documentclass{amsart}
\usepackage[dvips]{graphicx}
\usepackage{amscd}
\usepackage{pstricks}
\usepackage{amssymb}
\theoremstyle{plain}
\newtheorem{theorem}{Theorem}

\newtheorem{lemma}{Lemma}
\newtheorem{definition}{Definition}
\newtheorem{corollary}{Corollary}

\theoremstyle{remark}
\newtheorem{remark}{Remark}

\numberwithin{equation}{section}

\begin{document}

\title[]{A topological equivalence result for a family of nonlinear difference systems having generalized exponential dichotomy}
\author[]{\'Alvaro Casta\~neda}
\author[]{Gonzalo Robledo}
\address{Departamento de Matem\'aticas, Facultad de Ciencias, Universidad de
  Chile, Casilla 653, Santiago, Chile}
\email{castaneda@u.uchile.cl,grobledo@u.uchile.cl}
\thanks{The first author was funded by FONDECYT Iniciaci\'on Project 11121122 and by 
the Center of Dynamical Systems and Related Fields DySyRF (Anillo Project 1103, CONICYT). The second author
was funded by FONDECYT Regular Project 1120709}
\date{July 2015}

\begin{abstract}
We obtain sufficient conditions ensuring the topological equivalence of two perturbed difference linear systems 
whose linear part has a property of generalized exponential dichotomy. When the exponential dichotomy is verified,
we obtain a strongly and H\"older topological equivalence. 
\end{abstract}

\maketitle

\section{introduction}
The purpose of this article is to find sufficient conditions ensuring the topological equivalence (see Definition \ref{TopEq} in the next
section) between the difference systems
\begin{equation}
\label{No-lin}
x_{n+1}=A_{n}x_{n}+f(n,x_{n}),
\end{equation}
\begin{equation}
\label{No-lin2}
y_{n+1}=A_{n}y_{n}+g(n,y_{n}),
\end{equation}
where $x_{n}$ and $y_{n}$ are sequences of $d$--dimensional column vectors, $A_{n}\in \mathbb{R}^{d}\times\mathbb{R}^{d}$ and the 
functions $f,g\colon \mathbb{Z}\times \mathbb{R}^{d}\to \mathbb{R}^{d}$ satisfy
\begin{itemize}
\item[\textbf{(A1)}] $A_{n}$ is bounded, nonsingular and
\begin{displaymath}
||A_{n}-I||\leq M \quad \textnormal{for any} \quad n\in \mathbb{Z}, 
\end{displaymath}
where $||\cdot||$ is a matrix norm.
\item[\textbf{(A2)}] The functions $f$ and $g$ are in the set $\mathcal{S}$ defined by
\begin{displaymath}
\mathcal{S}=\Big\{ 
\mathcal{U}\colon \mathbb{Z}\times \mathbb{R}^{d}\to \mathbb{R}^{d}\colon 
|\mathcal{U}(n,x_{1})-\mathcal{U}(n,x_{2})|\leq r_{n}|x_{1}-x_{2}| \quad \textnormal{for any} \quad n\in \mathbb{Z}\Big\},
\end{displaymath}
where $|\cdot|$ is a vector norm and the sequence $r_{n}$ is nonnegative.
\end{itemize}

This problem was initially studied by G. Papaschinopoulos in \cite{Papaschinopoulos}, where the topological equivalence
of (\ref{No-lin}) and (\ref{No-lin2}) was an intermediate technical step in the study of the topological equivalence of some
hybrid systems. In \cite{Papaschinopoulos}, it was assumed that $f$ and $g$ satisfy some smallness assumptions, are Lipschitz
and the linear system
\begin{equation}
\label{Lin} 
z_{n+1}=A_{n}z_{n}
\end{equation}
has a property of $\alpha$--exponential dichotomy (see Definition \ref{ED} in the next section).

In this work, we consider more general assumptions compared with \cite{Papaschinopoulos}. In particular,
we assume that (\ref{Lin}) has a generalized exponential dichotomy (namely, a more general property) and obtain
sufficient conditions ensuring topological equivalence and strong topological equivalence. In addition, if (\ref{Lin}) 
has an $\alpha$--exponential dichotomy, we obtain sufficient conditions ensuring H\"older topological equivalence. In spite that our 
results are  strongly inspired in the works of Shi $\&$ Xiong \cite{Shi} and Jiang \cite{Jiang},\cite{Jian2} developed in 
the continuous case, the consequences obtained by our approach are not exactly the same ones.

The article is organized as follows: Section 2 introduces the main definitions (topological equivalence and exponential dichotomies).
Section 3 states the main results. Section 4 is devoted to several intermediate results. The proof
of the main results is developed in the section 5.

\section{Definitions}

The following definition has been introduced by Palmer \cite{Palmer} in the continuous case and extended to the discrete
case in a series of papers of Kurzweil, Papaschinopoulos and 
Schinas \cite{Kurzweil},\cite{Kurzweil-Papas},\cite{Schinas-Papas}:
\begin{definition}
\label{TopEq}
The systems \textnormal{(\ref{No-lin})} and \textnormal{(\ref{No-lin2})} are topologically equivalent
if there exists a map $H\colon \mathbb{Z}\times \mathbb{R}^{d}\to \mathbb{R}^{d}$ with the properties
\begin{itemize}
\item[(i)] For each fixed $n\in \mathbb{Z}$, the map $u\mapsto H(n,u)$ is an homeomorphism of $\mathbb{R}^{d}$.
\item[(ii)] $H(n,u)-u$ is bounded in $\mathbb{Z}\times \mathbb{R}^{d}$.
\item[(iii)] If $x_{n}$ is a solution of \textnormal{(\ref{No-lin})}, then $H[n,x_{n}]$ is a solution of  \textnormal{(\ref{No-lin2})}.
\end{itemize}
In addition, the map $u\mapsto L(n,u)=H^{-1}(n,u)$ has properties (i)--(iii) also.
\end{definition}

\begin{remark} Notice that the notation $H[n,x_{n}]$ is reserved to the special case when $x_{n}$ is a solution of (\ref{No-lin}).
On the other hand, the topological equivalence between (\ref{No-lin}) and (\ref{Lin}) can be defined in a similar way.
\end{remark}

The following definitions have been introduced by Shi and Xiong \cite{Shi} in the continuous case and we introduce its discrete version
\begin{definition} 
\label{strongly}
If the maps $u\mapsto H(n,u)$ and $u\mapsto L(n,u)$ are uniformly continuous for all $n\in \mathbb{Z}$ and satisfy properties
\textnormal{(i)--(iii)} of the previous definition, then we say that the systems 
\textnormal{(\ref{No-lin})} and \textnormal{(\ref{No-lin2})} are strongly topologically equivalent.
\end{definition}

\begin{definition}
\label{holder}
If the maps $u\mapsto H(n,u)$ and $u\mapsto L(n,u)$ are H\"older continuous for all $n\in \mathbb{Z}$ and satisfy properties
\textnormal{(i)--(iii)} of the previous definition, then we say that the systems 
\textnormal{(\ref{No-lin})} and \textnormal{(\ref{No-lin2})} are H\"older topologically equivalent.
\end{definition}

The problem of the topological equivalence has been extensively studied in the continuous non--autonomous case for
several authors, which follow the seminal paper of Palmer \cite{Palmer}. We pay special atention to the works
of \cite{Chen},\cite{Jiang},\cite{Jian2}, which use the concept of generalized exponential dichotomy introduced by
Martin \cite{Martin}.

Before to introduce the next definitions, we will denote the fundamental matrix of (\ref{Lin}) by $W_{n}$ (\emph{i.e.}, $W_{n+1}=A_{n}W_{n}$).

The generalized exponential dichotomy in a discrete context is defined as follows: 
\begin{definition}
\label{GDD}
The system \textnormal{(\ref{Lin})} has a generalized exponential dichotomy if there exists a projection $P$ ($P^{2}=P$), a constant $K\geq 1$ and 
a non--negative sequence $\{a_{n}\}_{n\in \mathbb{Z}}$ satisfying
\begin{equation}
\label{condition1}
\sum\limits_{j=p}^{q}a_{j}\to +\infty \quad \textnormal{as} \quad q\to +\infty \quad \textnormal{for fixed $p\in \mathbb{Z}$}, 
\end{equation}
\begin{equation}
\label{condition2}
\sum\limits_{j=p}^{q}a_{j}\to +\infty \quad \textnormal{as} \quad p\to -\infty \quad \textnormal{for fixed $q\in \mathbb{Z}$}
\end{equation}
such that 
\begin{equation}
 \label{GDD-def}
\left\{\begin{array}{rcl}
  ||W_{n}PW_{m}^{-1}||\leq K\exp\Big(-\sum\limits_{j=m}^{n}a_{j}\Big) & \textnormal{if}& n\geq m\\
   ||W_{n}(I-P)W_{m}^{-1}||\leq K\exp\Big(-\sum\limits_{j=n}^{m}a_{j}\Big) & \textnormal{if}& n< m.
   \end{array}\right.
\end{equation}
\end{definition}

It is interesting to observe that (\ref{condition1})--(\ref{condition2}) are satisfied in the case
$a_{j}=\alpha>0$ for any $j\in \mathbb{Z}$, which leads to the classic definition of $\alpha$--exponential 
dichotomy:
\begin{definition}
\label{ED}
The system \textnormal{(\ref{Lin})} has an $\alpha$--exponential dichotomy if there exists a projection $P$ ($P^{2}=P$), a 
constant $K\geq 1$ and 
\begin{equation}
 \label{ED-def}
\left\{\begin{array}{rcl}
  ||W_{n}PW_{m}^{-1}||\leq Ke^{-\alpha(n-m)} & \textnormal{if}& n\geq m\\
   ||W_{n}(I-P)W_{m}^{-1}||\leq Ke^{-\alpha(m-n)} & \textnormal{if}& n< m.
   \end{array}\right.
\end{equation}
\end{definition}

\begin{remark}
\label{relacion-exp}
The notation (\ref{ED-def}) was taken from \cite[p.165]{Papaschinopoulos} but other equivalent notations have been 
introduced in \cite{Kurzweil} and \cite{Papas-Schi}. For a deeper discussion about discrete dichotomies, we refer
the reader to \cite{Coffman} and \cite{Pinto}.
\end{remark}

The following example shows a linear system having a generalized exponential dichotomy but not an exponential one: let us 
consider (\ref{Lin}) with a matrix 
\begin{displaymath}
A_{n}=\left[\begin{array}{cc}
              b_{n} & 0\\
              0     & 1/b_{n}
             \end{array}\right],
\end{displaymath}
where $0<b_{n}=b_{-n}<1$ for any $n\in \mathbb{Z}$, $b_{n}\to 1$ monotonically as $n\to +\infty$ and (\ref{condition1})--(\ref{condition2}) are satisfied
for $a_{j}=|\ln(b_{j})|$. 

Notice that this system has the generalized exponential dichotomy since
\begin{displaymath}
W_{n}=\left[\begin{array}{cc}
              \prod\limits_{j=0}^{n-1}b_{j} & 0\\
              0     & \displaystyle \prod\limits_{j=0}^{n-1}\frac{1}{b_{j}}
             \end{array}\right]
\quad \textnormal{with} \quad
P=\left[\begin{array}{cc}
              1 & 0\\
              0     & 0
             \end{array}\right]
\end{displaymath}
leads to (\ref{GDD-def}) with $K=1$ and $a_{j}=|\ln(b_{j})|$. Nevertheless, let us observe that the system has not
an exponential dichotomy. Indeed, otherwise, there exists $\alpha>0$ such that
\begin{displaymath}
\sum\limits_{k=m}^{n}|\ln(b_{k})|\geq \alpha(n-m), \quad \textnormal{for any} \quad n\geq m,
\end{displaymath}
then, when considering $n=m+T$ (for some $T\in \mathbb{N}$), it follows that 
\begin{displaymath}
\frac{1}{T}\sum\limits_{k=m}^{m+T}|\ln(b_{k})|\geq \alpha, \quad \textnormal{for any} \quad m\in \mathbb{Z}.
\end{displaymath}
Now, we obtain a contradiction by letting $m\to +\infty$.

\begin{remark}
Notice that (\ref{GDD-def}) can be viewed in terms of the Green function:
\begin{equation}
 \label{Green}
G(n,m)=\left\{\begin{array}{ccl}
  W_{n}PW_{m}^{-1} & \textnormal{if}& n\geq m\\
  -W_{n}(I-P)W_{m}^{-1} & \textnormal{if}& n<m.
   \end{array}\right.
\end{equation}   
\end{remark}

\begin{definition}
For any sequence $g_{n}$ ($n\in \mathbb{Z}$), let us define the map 
\begin{displaymath}
N(n,g)= \sum\limits_{m=-\infty}^{n-1}K\exp\Big(-\sum\limits_{j=m+1}^{n}a_{j}\Big)g_{m}+
        \sum\limits_{m=n}^{\infty}K\exp\Big(-\sum\limits_{j=n}^{m+1}a_{j}\Big)g_{m},
\end{displaymath}
where $K$ and $a_{j}$ are stated in Definition \textnormal{\ref{GDD}}.
\end{definition}

\section{Main Results}

\begin{theorem}
\label{eq-top-f-g}
Suppose that \textnormal{(\ref{Lin})} has a generalized exponential dichotomy and the functions
$f$ and $g$ satisfy
\begin{itemize}
 \item[\textbf{(H1)}] $|f(n,x)|\leq F_{n}$ and $|g(n,x)|\leq G_{n}$ where $F_{n}$ and $G_{n}$ are nonnegative sequences.
 \item[\textbf{(H2)}] There exists $B>0$ such that the sequences $F_{n}$ and $G_{n}$ verify
 \begin{equation}
  \label{lip-GF}
 N(n,G+F)\leq B.
  \end{equation}
 \item[\textbf{(H3)}]  There exists $\theta\in (0,1)$ such that the sequence $r_{n}$ stated in \textnormal{\textbf{(A2)}} satisfies
\begin{equation}
\label{lip-r}
N(n,r)\leq \theta<1,
\end{equation}
 \item[\textbf{(H4)}] For any $(u,u',x,x')\in \mathbb{R}^{d}\times\mathbb{R}^{d}\times\mathbb{R}^{d}\times\mathbb{R}^{d}$ with $|u|,|u'|\leq B$, the function
\begin{displaymath}
\sum\limits_{k=-\infty}^{n-1-J}K\exp\Big(-\sum\limits_{p=k+1}^{n}a_{p}\Big)|\Delta_{k}(u,u',x,x')|+
\sum\limits_{k=n+J}^{\infty}K\exp\Big(-\sum\limits_{p=n}^{k+1}a_{p}\Big)|\Delta_{k}(u,u',x,x')|
\end{displaymath}
with $\Delta_{k}$ defined by
\begin{displaymath}
\Delta_{k}(u,u',x,x')=g(k,u+x)-g(k,u'+x')+f(k,x')-f(k,x),
\end{displaymath}
converges uniformly on $(u,u',x,x')$ to zero when $J\to + \infty$.
\item[\textbf{(H5)}] For any $(v,v',y,y')\in \mathbb{R}^{d}\times\mathbb{R}^{d}\times\mathbb{R}^{d}\times \mathbb{R}^{d}$ with $|v|,|v'|\leq B$, the function
\begin{displaymath}
\sum\limits_{k=-\infty}^{n-1-J}K\exp\Big(-\sum\limits_{p=k+1}^{n}a_{p}\Big)|\overline{\Delta}_{k}(v,v',y,y')|+
\sum\limits_{k=n+J}^{\infty}K\exp\Big(-\sum\limits_{p=n}^{k+1}a_{p}\Big)|\overline{\Delta}_{k}(v,v',y,y')|
\end{displaymath}
with $\overline{\Delta}_{k}$ defined by
\begin{displaymath}
\begin{array}{rcl}
\overline{\Delta}_{k}(v,v',y,y')&=&f(k,v+y)-f(k,v'+y')+g(k,y)-g(k,y'),
\end{array}
\end{displaymath} 
converges uniformly on $(v,v',y,y')$ to zero when $J\to +\infty$,
\end{itemize}
then \textnormal{(\ref{No-lin})} and \textnormal{(\ref{No-lin2})} are topologically equivalent.
\end{theorem}

\begin{remark}
\label{comments}
A continuous version of this theorem has been studied by Chen $\&$ Xia \cite{Chen} and Jiang \cite{Jiang}, this last,
considering $g(\cdot,\cdot)=0$. As in \cite[Theorem 2.2]{Chen}, we obtain a topologically equivalence result. Neverthless,
in \cite[Theorem 2]{Jian2} a result of strong topological equivalence is obtianed. We will explain
this point in the proof.

\textbf{(H1)} is a technical assumption which generalizes the case studied by 
Pa\-pas\-chi\-no\-pou\-los \cite{Papaschinopoulos}, where it is assumed that $|f(n,x)|$ and $|g(n,x)|$ are bounded by a 
small enough positive constant. We emphasize that $F_{n}$ and $G_{n}$ are not necessarily bounded sequences. 

\textbf{(H2)} is introduced in order to ensure that if (\ref{Lin}) is perturbed by linear combinations 
of $f$ and $g$, then the corresponding perturbed systems has
a unique bounded solution. Altough $F_{n}$ and $G_{n}$ could be unbounded sequences, \textbf{(H2)} says
that they must be dominated by terms $\exp(-\sum a_{n})$ at $\pm \infty$.

\textbf{(H3)} is usual in the topological equivalence literature and plays a key role in several 
intermediate steps as the proof of the continuity of the map $u\mapsto H(n,u)$ and the use of the Banach fixed point.
As before, $r_{n}$ is not necessarily a bounded sequence but  must be dominated by terms $\exp(-\sum a_{n})$
at $\pm \infty$.

\textbf{(H4)} and \textbf{(H5)} are introduced in order to prove the continuity of the maps $u\mapsto H(n,u)$
and $u\mapsto H^{-1}(n,u)$.  In spite of \textbf{(H2)} ensures that the corresponding limits are zero 
when $J\to +\infty$, the rate of convergence is not necessarily uniform, which is ensured by these hypotheses. It is
important to emphasize that if $g(\cdot,\cdot)=0$, these assumptions can be seen as the discrete version of a technical condition
introduced by Jiang in Theorem 2 from \cite{Jiang}.
\end{remark}

\begin{remark}
In the case $g(t,\cdot)=0$, we can obtain simpler conditions ensuring that (\ref{No-lin}) and (\ref{Lin}) are topologically
equivalent.
\end{remark}

\begin{corollary}
\label{cor2}
Suppose that \textnormal{(\ref{Lin})} has a generalized exponential dichotomy and the functions $f$ and $g$ satisfy \textnormal{\textbf{(H1)}}--
\textnormal{\textbf{(H5)}}. If $\{r_{n}\}$ verifies
\begin{equation}
\label{Stepanov}
\sup\limits_{n\in \mathbb{Z}}\frac{1}{2L}\sum\limits_{k=n-L}^{n+L}r_{k}<M_{0},
\end{equation}
then \textnormal{(\ref{No-lin})} is strongly topologically equivalent to \textnormal{(\ref{No-lin2})}.
\end{corollary}

\begin{remark}
The left side of (\ref{Stepanov}) can be seen as a discrete Stepanov's norm (see \emph{e.g.}, \cite{Besicovitch}).  
In addition, (\ref{Stepanov}) is always satisfied when $\{r_{k}\}_{k}\in \ell_{\infty}(\mathbb{Z})$.
\end{remark}

As stated above, if $a_{n}=\alpha>0$, then (\ref{Lin}) has an $\alpha$--exponential dichotomy. In addition, if $F_{n}$, $G_{n}$
and $r_{n}$ are also positive constants (namely, $F$,$G$ and $r$), then \textbf{(H4)} and \textbf{(H5)} are immediately satisfied
since $|\Delta_{k}|$ and $|\overline{\Delta}_{k}|$ are bounded by $2(F+G)$ for any $k\in \mathbb{Z}$ and 
$$
\sum\limits_{k=-\infty}^{n-1-J}e^{-\alpha(n-k-1)} \quad \textnormal{and} \quad \sum\limits_{k=n+J}^{+\infty}e^{-\alpha(k+1-n)}
$$
converge to zero when $J\to +\infty$ and the rate of convergence is independent of the points. This allows to formulate:

\begin{theorem}
\label{eq-top-f-g2}
Suppose that \textnormal{(\ref{Lin})} has an $\alpha$--exponential dichotomy and the functions
$f$ and $g$ satisfy
\begin{itemize}
\item[\textbf{(D1)}] $|f(n,x)|\leq F$ and $|g(n,x)|\leq G$ where $F$ and $G$ are nonnegative constants.
\item[\textbf{(D2)}]  The functions $f$ and $g$ are in the set $\mathcal{S}'$ defined by
\begin{displaymath}
\mathcal{S}'=\Big\{ 
\mathcal{U}\colon \mathbb{Z}\times \mathbb{R}^{d}\to \mathbb{R}^{d}\colon 
|\mathcal{U}(n,x_{1})-\mathcal{U}(n,x_{2})|\leq r|x_{1}-x_{2}| \quad \textnormal{for any} \quad n\in \mathbb{Z}\Big\},
\end{displaymath}
where $r>0$ is such that
\begin{equation}
\theta = Kr\frac{1+e^{-\alpha}}{1-e^{-\alpha}}<1, 
\end{equation}
then \textnormal{(\ref{No-lin})} and \textnormal{(\ref{No-lin2})} are strongly topologically equivalent. 

\noindent Moreover, if
$M+r<\alpha$, then  \textnormal{(\ref{No-lin})} and \textnormal{(\ref{No-lin2})} are H\"older topologically equivalent.
\end{itemize}
\end{theorem}

\section{Preliminar Results}
\begin{lemma}
\label{zero-sol}
If \textnormal{(\ref{Lin})} has a generalized exponential dichotomy, then the unique solution of \textnormal{(\ref{Lin})}
bounded on $\mathbb{Z}$ is $y_{n}=0$.
\end{lemma}

\begin{proof}
As in \cite[p.11]{Coppel}, it is easy to verify that (\ref{GDD-def}) implies 
\begin{displaymath}
\begin{array}{rcl}
||W_{n}P\xi||\leq K\exp\Big(-\sum\limits_{j=m}^{n}a_{j}\Big)||W_{m}P\xi|| & \textnormal{if} & n\geq m\\
||W_{n}(I-P)\xi||\leq K\exp\Big(-\sum\limits_{j=n}^{m}a_{j}\Big)||W_{m}(I-P)\xi|| & \textnormal{if}& n< m.
\end{array}
\end{displaymath}
for any initial condition $\xi\in \mathbb{R}^{d}$. In addition, let us assume that the projection $P$ has rank $k\leq d$. 

The first inequality above is equivalent to
\begin{displaymath}
\frac{1}{K}\exp\Big(\sum\limits_{j=m}^{n}a_{j}\Big) ||W_{n}P\xi||\leq ||W_{m}P\xi|| \quad \textnormal{if} \quad n\geq m.
\end{displaymath}

By using (\ref{condition2}), we can see that there exists a $k$--dimensional subspace of initial conditions
leading to solutions tending to the infinite when $m\to -\infty$.

On the other hand, the second inequality is equivalent to
\begin{displaymath}
\frac{1}{K}\exp\Big(\sum\limits_{j=n}^{m}a_{j}\Big)||W_{n}(I-P)\xi||\leq ||W_{m}(I-P)\xi|| \quad \textnormal{if} \quad n< m.
\end{displaymath}

As before, by (\ref{condition1}), we can see that there exists a $(d-k)$--dimensional subspace 
of initial conditions leading to solutions tending to the infinite when $m\to +\infty$. In consequence,
the unique bounded solution can be the trivial one.
\end{proof}

\begin{lemma}
\label{bounded}
If \textnormal{(\ref{Lin})} has a generalized exponential dichotomy and a sequence $q_{n}$
verifies 
\begin{itemize}
\item[\textbf{(E1)}] $\sup\limits_{n\in \mathbb{Z}}|N(n,|q|)|<+\infty$,
\end{itemize}
then the system
\begin{equation}
\label{auxiliary}
z_{n+1}=A_{n}z_{n}+q_{n}
\end{equation}
has a unique bounded solution given by
\begin{displaymath}
\hat{\phi}_{n}=\sum\limits_{m=-\infty}^{\infty}G(n,m+1)q_{m}.
\end{displaymath}

\end{lemma}

\begin{proof}The proof has two steps:\\

\noindent\emph{Boundedness of $\hat{\phi}_{n}$:} It is straightforward (see \emph{e.g.}, \cite{Elaydi}) to see that $\hat{\phi}_{n}$ is solution of (\ref{auxiliary}). In 
order to verify that $\hat{\phi}_{n}$ is bounded, notice that:
\begin{displaymath}
\begin{array}{rcl}
|\hat{\phi}_{n}|& \leq &\sum\limits_{m=-\infty}^{n-1}|G(n,m+1)q_{m}|+\sum\limits_{m=n}^{\infty}|G(n,m+1)q_{m}|\\\\
        &=&\sum\limits_{m=-\infty}^{n-1}|W_{n}PW_{m+1}^{-1}q_{m}|+\sum\limits_{m=n}^{\infty}|W_{n}(I-P)W_{m+1}^{-1}q_{m}|\\
        &\leq&\sum\limits_{m=-\infty}^{n-1}K\exp\Big(-\sum\limits_{j=m+1}^{n}a_{j}\Big)|q_{n}|+
        \sum\limits_{m=n}^{\infty}K\exp\Big(-\sum\limits_{j=n}^{m+1}a_{j}\Big)|q_{n}|\\\\
        &=&N(n,|q|)
\end{array}
\end{displaymath}
and the boundedness follows from \textbf{(E1)}.

\noindent\emph{Uniqueness of the bounded solution:} As in \cite{Chen} (continuous framework), let $y_{n}$ be 
another bounded solution of (\ref{auxiliary}). By variation of parameters (see \emph{e.g.} \cite[Th. 3.17]{Elaydi}), we know that
\begin{displaymath}
\begin{array}{rcl}
y_{n}&=&W_{n}W_{0}^{-1}y_{0}+\sum\limits_{r=0}^{n-1}W_{n}W_{r+1}^{-1}q_{r}\\\\
     &=&W_{n}W_{0}^{-1}y_{0}+\sum\limits_{r=0}^{n-1}W_{n}PW_{r+1}^{-1}q_{r}+\sum\limits_{r=0}^{n-1}W_{n}(I-P)W_{r+1}^{-1}q_{r}\\\\
     &=&W_{n}W_{0}^{-1}y_{0}+\sum\limits_{r=-\infty}^{n-1}W_{n}PW_{r+1}^{-1}q_{r}-\sum\limits_{r=-\infty}^{-1}W_{n}PW_{r+1}^{-1}q_{r}\\\\
     & &+\sum\limits_{r=0}^{\infty}W_{n}(I-P)W_{r+1}^{-1}q_{r}-\sum\limits_{r=n}^{\infty}W_{n}(I-P)W_{r+1}^{-1}q_{r}.
     \end{array}
\end{displaymath}

It is important to note that the expression above is well defind because
\begin{displaymath}
\begin{array}{rcl}
\left|\sum\limits_{r=-\infty}^{-1}W_{n}PW_{r+1}^{-1}q_{r}\right|&=&\left|W_{n}W_{0}^{-1}\sum\limits_{r=-\infty}^{-1}W_{0}PW_{r+1}^{-1}q_{r}\right|  \\\\
                        &\leq & \left| W_{n}W_{0}^{-1}\right| \sum\limits_{r=-\infty}^{-1}|W_{0}PW_{r+1}^{-1}q_{r}|\\\\
                        &\leq & \left|W_{n}W_{0}^{-1}\right|\sum\limits_{r=-\infty}^{-1}K\exp\Big(-\sum\limits_{j=r}^{-1}a_{j}\Big)|q_{r}|\\\\
                        &\leq & |W_{n}W_{0}^{-1}| N(r,|q|)
\end{array}
                        \end{displaymath}
and let us denote
\begin{displaymath}
 \sum\limits_{r=-\infty}^{-1}W_{n}PW_{r+1}^{-1}q_{r}=W_{n}W_{0}^{-1}y_{1}.
\end{displaymath}

In a similar way, we can verify that
\begin{displaymath}
\sum\limits_{r=n}^{\infty}W_{n}(I-P)W_{r+1}^{-1}q_{r}=W_{n}W_{0}^{-1}y_{2}.
\end{displaymath}

Now, we can see that
\begin{displaymath}
\begin{array}{rcl}
y_{n}&=&W_{n}W_{0}^{-1}(y_{0}-y_{1}+y_{2})+\sum\limits_{r=-\infty}^{n-1}W_{n}PW_{r+1}^{-1}q_{r}-\sum\limits_{r=n}^{\infty}W_{n}(I-P)W_{r+1}^{-1}q_{r}.
\end{array}
\end{displaymath}

As $y_{n}$ is a bounded solution of (\ref{auxiliary}) and \textbf{(E1)} implies that
$$
\sum\limits_{r=-\infty}^{n-1}W_{n}PW_{r+1}^{-1}q_{r}-\sum\limits_{r=n}^{\infty}W_{n}(I-P)W_{r+1}^{-1}q_{r}
$$
is also bounded, it follows that $x_{n}=W_{n}W_{0}^{-1}(y_{0}-y_{1}+y_{2})$ is a bounded 
solution of (\ref{Lin}). Finally, Lemma \ref{zero-sol} implies that $y_{0}=y_{1}-y_{2}$ and the uniqueness follows.
\end{proof}

\begin{lemma}
\label{bounded-Q}
If \textnormal{(\ref{Lin})} has a generalized exponential dichotomy and the system
\begin{equation}
\label{auxiliary-Q}
z_{n+1}=A_{n}z_{n}+q(n,z_{n})
\end{equation}
is such that 
\begin{equation}
\label{Cota-Q}
|q(n,z)|\leq Q_{n} \quad \textnormal{and} \quad |q(n,z)-q(n,\tilde{z})|\leq r_{n}|z-\tilde{z}|,
\end{equation}
where $Q_{n}$ and $r_{n}$ satisfy
\begin{equation}
\label{CTQ}
N(n,Q)\leq \tilde{B} \quad \textnormal{and} \quad N(n,r)\leq \theta <1,
\end{equation}
then, there exists a unique bounded solution of \textnormal{(\ref{auxiliary-Q})}. 
\end{lemma}

\begin{proof}
\noindent\emph{Existence:} Let us consider the sequence $\{\varphi^{(j)}\}_{j}$, recursively defined
by \begin{displaymath}
\varphi_{n+1}^{(j)}=A_{n}\varphi_{n}^{(j)}+q(n,\varphi_{n}^{(j-1)}),
\end{displaymath}
where $\varphi^{(0)}$ is an arbitrary sequence in $\ell_{\infty}(\mathbb{Z})$ satisfying $|\varphi^{(0)}|_{\infty}\leq \tilde{B}$.

By using Lemma \ref{bounded} combined with the first inequalities of (\ref{Cota-Q})--(\ref{CTQ}), we can see that 
$\varphi^{(j)}$ is the unique solution of the above system and verifies
\begin{displaymath}
\varphi_{n}^{(j)}=\sum\limits_{k=-\infty}^{+\infty}G(n,k+1)q(k,\varphi_{k}^{(j-1)}), 
\end{displaymath}
with $|\varphi^{(j)}|_{\infty}\leq \tilde{B}$ for any $j\in \mathbb{N}$.

On the other hand, the second inequalities of (\ref{Cota-Q})--(\ref{CTQ}) imply that 
$$
|\varphi^{(j)}-{\varphi}^{(j-1)}|_{\infty}\leq \theta|\varphi^{(j-1)}-\varphi^{(j-2)}|_{\infty}
$$
with $\theta\in (0,1)$, and we can see that $\varphi^{(j)}$ is a Cauchy sequence in $\ell_{\infty}(\mathbb{Z})$. Now, letting $j\to +\infty$
in $\varphi_{n}^{(j)}$, it follows that
$$
\varphi_{n}^{*}=\sum\limits_{k=-\infty}^{+\infty}G(n,k+1)q(k,\varphi_{k}^{*}), 
$$
is a bounded solution of (\ref{auxiliary-Q}).

\noindent\emph{Uniqueness:} Let $y_{n}$ be another bounded solution of (\ref{auxiliary-Q}). By following the lines of the proof
of Lemma \ref{bounded} combined with (\ref{Cota-Q})--(\ref{CTQ}), the reader can verify that
$$
y_{n}=\sum\limits_{k=-\infty}^{+\infty}G(n,k+1)q(k,y_{k}).
$$

Finally, by using the second inequalities of (\ref{Cota-Q})--(\ref{CTQ}), we have that
$$
|\varphi^{*}-y|_{\infty}\leq \theta |\varphi^{*}-y|_{\infty}
$$
and the uniqueness follows since $0<\theta<1$.
\end{proof}
\begin{lemma}
\label{lemme-0}
Suppose that \textnormal{(\ref{Lin})} has a generalized exponential dichotomy. If the systems
\textnormal{(\ref{No-lin})}--\textnormal{(\ref{No-lin2})} satisfy \textnormal{\textbf{(H1)}--\textbf{(H3)}}
and $x(n,m,\xi)$ is the solution of \textnormal{(\ref{No-lin})} with initial condition $\xi$ at $n=m$, then 
the $(m,\xi)$--parameter dependent system
\begin{equation}
\label{auxiliary2}
w_{n+1}=A_{n}w_{n}-f(n,x(n,m,\xi))+g(n,w_{n}+x(n,m,\xi)).
\end{equation}
has a unique bounded solution $n\mapsto \chi(n;(m,\xi))$ with $|\chi(n;(m,\xi))|_{\infty}\leq B$.
\end{lemma}

\begin{proof}
By using \textbf{(H1)}--\textbf{(H3)} and Lemma \ref{bounded-Q} with $q(n,w_{n})=-f(n,x(n,m,\xi))+g(n,w_{n}+x(n,m,\xi))$, we know that the unique bounded solution of (\ref{auxiliary2})
is
\begin{equation}
\label{cauchy2}
\chi(n;(m,\xi))=\sum\limits_{k=-\infty}^{+\infty}G(n,k+1)\{g(k,\chi(k;(m,\xi))+x_{k,m}(\xi))-f(k,x_{k,m}(\xi))\}, 
\end{equation}
where $x_{k,m}(\xi)=x(k,m,\xi)$ and the Lemma follows.
\end{proof}


\begin{lemma}
\label{lemme-00}
Suppose that \textnormal{(\ref{Lin})} has a generalized exponential dichotomy. If the systems
\textnormal{(\ref{No-lin})}--\textnormal{(\ref{No-lin2})} satisfy \textnormal{\textbf{(H1)}--\textbf{(H3)}}
and $y(n,m,\nu)$ is the solution of \textnormal{(\ref{No-lin2})} with initial condition $\nu$ at $n=m$, then 
the $(m,\nu)$--parameter dependent system
\begin{equation}
\label{auxiliary3}
z_{n+1}=A_{n}z_{n}+f(n,z_{n}+y(n,m,\nu))-g(n,y(n,m,\nu)),
\end{equation}
has a unique bounded solution $n\mapsto \vartheta(n;(m,\nu))$ with $|\vartheta(n;(m,\nu))|_{\infty} \leq B$. 
\end{lemma}

\begin{proof}
As before, by using \textbf{(H1)}--\textbf{(H3)} and Lemma \ref{bounded-Q} with $q(n,z_{n})=f(n,z_{n}+y(n,m,\nu))-g(n,y(n,m,\nu))$, the unique
bounded solution of (\ref{auxiliary3}) is
\begin{equation}
\label{cauchy3}
\vartheta(n;(m,\nu))=\sum\limits_{k=-\infty}^{+\infty}G(n,k+1)\{f(k,\vartheta(k;(m,\nu))+y_{k,m}(\nu))-g(k,y_{k,m}(\nu))\}, 
\end{equation}
where $y_{k,m}(\nu)=y(k,m,\nu)$.
\end{proof}

\begin{remark}
\label{PIU}
By uniqueness of the solution of (\ref{No-lin}), we know that $x(n,n,x(n,m,\xi))=x(n,m,\xi)$, which implies that
(\ref{auxiliary2}) is similar to
\begin{displaymath}
w_{n+1}=A_{n}w_{n}-f(n,x(n,n,x(n,m,\xi)))+g(n,w_{n}+x(n,n,x(n,m,\xi)))
\end{displaymath}
and Lemma \ref{lemme-0} implies that
\begin{equation}
\label{identidad-fun}
\chi(n;(m,\xi))=\chi(n;(n,x(n,m,\xi))).
\end{equation}

In a similar way, it can be proved that
\begin{equation}
\label{identidad-fun2}
\vartheta(n;(m,\nu))=\vartheta(n;(n,y(n,m,\nu))).
\end{equation}
\end{remark}

\begin{lemma}
\label{lemme-1}
Suppose that \textnormal{(\ref{Lin})} has a generalized exponential dichotomy. If the systems
\textnormal{(\ref{No-lin})}--\textnormal{(\ref{No-lin2})} satisfy \textnormal{\textbf{(H1)}--\textbf{(H3)}}, then 
there exists a unique map $H\colon \mathbb{Z}\times \mathbb{R}^{d}\to \mathbb{R}^{d}$, which verifies the following properties
\begin{itemize}
 \item[a)] $H(n,\xi)-\xi$ is bounded for any fixed $n\in \mathbb{Z}$ and $\xi\in \mathbb{R}^{d}$.
 \item[b)] If $x_{n}=x(n,m,\xi)$ is solution of \textnormal{(\ref{No-lin})}, then $H[n,x_{n}]$ is solution 
 of \textnormal{(\ref{No-lin2})}.
\end{itemize}
\end{lemma}

\begin{proof} The proof will be divided in two steps:

\noindent\textit{Step i: Existence of $H$.} We will prove that 
\begin{displaymath}
H(n,\xi)=\xi+\chi(n;(n,\xi)) 
\end{displaymath}
satisfy properties a) and b).

Indeed, by using (\ref{cauchy2}) combined with \textbf{(H1)--(H2)}, we obtain that $|H(n,\xi)-\xi|\leq B$. On the other hand,
we replace  $(n,\xi)$ by $\big(n,x(n,m,\xi)\big)$ and (\ref{identidad-fun}) implies
\begin{displaymath}
\begin{array}{rcl}
H[n,x(n,m,\xi)]&=&x(n,m,\xi)+\chi\big(n;(n,x(n,m,\xi))\big)\\\\
               &=&x(n,m,\xi)+\chi\big(n;(m,\xi)\big)
\end{array}               
\end{displaymath}
and the reader can verify easily that $H[n,x(n,m,\xi)]$ is solution of (\ref{No-lin2}) since $n\mapsto \chi\big(n;(m,\xi)\big)$
is solution of (\ref{auxiliary2}).

\noindent\textit{Step ii: Uniqueness of $H$.} Let $\tilde{H}$ be another map satisfying a) and b). Let us
observe that $u_{n}=\tilde{H}[n,x_{n}]-x_{n}$ is also a bounded solution of (\ref{auxiliary2}), which implies
by Lemma \ref{lemme-0} that $\tilde{H}[n,x_{n}]-x_{n}=\chi\big(n;(m,\xi)\big)$ and the uniqueness follows.
\end{proof}

\begin{lemma}
\label{lemme-2}
Suppose that \textnormal{(\ref{Lin})} has a generalized exponential dichotomy. If the systems
\textnormal{(\ref{No-lin})}--\textnormal{(\ref{No-lin2})} satisfy \textnormal{\textbf{(H1)}--\textbf{(H3)}}, then 
there exists a unique map $L\colon \mathbb{Z}\times \mathbb{R}^{d}\to \mathbb{R}^{d}$, which verifies the following properties
\begin{itemize}
 \item[a)] $L(n,\nu)-\nu$ is bounded for any fixed $n\in \mathbb{Z}$ and $\nu\in \mathbb{R}^{d}$.
 \item[b)] If $y_{n}=y(n,m,\nu)$ is solution of \textnormal{(\ref{No-lin2})}, then $L[n,y_{n}]$ is solution 
 of \textnormal{(\ref{No-lin})}.
\end{itemize}
\end{lemma}

\begin{proof}
It can be proved analogously as the previous result that the map
\begin{displaymath}
L(n,\nu)=\nu+\vartheta(n;(n,\nu)) 
\end{displaymath}
is the unique satisfying properties a) and b). 
\end{proof}

\begin{remark}
\label{reescritura}
By Lemma \ref{lemme-1}
combined with (\ref{cauchy2}), we know that $H[n,x(n,m,\xi)]$ can be written as follows:
\begin{equation}
\label{hom1}
\begin{array}{rcl}
H[n,x(n,m,\xi)]&=&\sum\limits_{k=-\infty}^{+\infty}G(n,k+1)g(k,H[k,x(k,m,\xi)])\\\\
&&-\sum\limits_{k=-\infty}^{+\infty}G(n,k+1)f(k,x(k,m,\xi))+x(n,m,\xi).
\end{array}
\end{equation}

Similarly, by Lemma \ref{lemme-2} combined with (\ref{cauchy3}), we know that:
\begin{equation}
\label{hom2}
\begin{array}{rcl}
L[n,y(n,m,\nu)]&=&\sum\limits_{k=-\infty}^{+\infty}G(n,k+1)f(k,L[k,y(k,m,\nu)])\\\\
&&-\sum\limits_{k=-\infty}^{+\infty}G(n,k+1)g(k,y(k,m,\nu))+y(n,m,\nu).
\end{array}
\end{equation}
\end{remark}

\begin{lemma}
For any solution $x(n,m,\xi)$ of \textnormal{(\ref{No-lin})} and $y(n,m,\nu)$ of \textnormal{(\ref{No-lin2})} and
any $n\in \mathbb{Z}$, it follows that
\begin{displaymath}
L[n,H[n,x(n,m,\xi)]]=x(n,m,\xi) \quad \textnormal{and} \quad H[n,L[n,y(n,m,\nu)]]=y(n,m,\nu).
\end{displaymath}
\end{lemma}

\begin{proof}
By Lemma \ref{lemme-1}
and Remark \ref{reescritura}, we know that (\ref{hom1}) is solution of (\ref{No-lin2}). Now, by Lemma \ref{lemme-2}, we also 
know that $L[n,H[n,x_{n}(\xi)]]$ is a solution of (\ref{No-lin}) that can be written as follows:
\begin{displaymath}
\begin{array}{rcl}
L[n,H[n,(n,m,\xi)]]&=&V[n,x(n,m,\xi)]\\\\
&=&\sum\limits_{k=-\infty}^{+\infty}G(n,k+1)f(k,V[k,x(k,m,\xi)])\\\\
&&-\sum\limits_{k=-\infty}^{+\infty}G(n,k+1)g(k,H[k,x(k,m,\xi)])+H[n,x(n,m,\xi)].
\end{array}
\end{displaymath}

Now, by using (\ref{hom1}) combined with \textbf{(A2)}, we can deduce that
\begin{displaymath}
\begin{array}{rcl}
|V[n,x(n,m,\xi)]-x(n,m,\xi)|&\leq &\sum\limits_{k=-\infty}^{+\infty}|G(n,k+1)|\\\\
&& |f(k,V[k,x(k,m,\xi)])-f(k,x(k,m,\xi))|\\\\
&\leq & \sum\limits_{k=-\infty}^{+\infty}|G(n,k+1)|r_{k}|V[k,x(k,m,\xi)]-x(k,m,\xi)|\\\\
\end{array}
\end{displaymath}
and \textbf{(H3)} implies that
$$
|L[n,H[n,x(n,m,\xi)]]-x(n,m,\xi)|_{\infty}\leq \theta |L[n,H[n,x(n,m,\xi)]]-x(n,m,\xi)|_{\infty},
$$
with $\theta\in (0,1)$, which is equivalent to 
\begin{equation}
\label{Hom-L}
L[n,H[n,x(n,m,\xi)]]=x(n,m,\xi).
\end{equation}

In a similar way, the reader can verify that
\begin{equation}
\label{Hom-H}
H[n,L[n,y(n,m,\nu)]]=y(n,m,\nu).
\end{equation}

\end{proof}

\begin{remark}
\label{almost}
The maps $\xi \mapsto H(n,\xi)$  and $\nu \mapsto L(n,\nu)$ defined by
\begin{displaymath}
\begin{array}{rcl}
H(n,\xi)&=&\xi + \chi(n;(n,\xi)) \\\\
&=&\xi+\hspace{-0.05cm}\sum\limits_{k=-\infty}^{+\infty}G(n,k+1)\{g(k,\chi(k;(n,\xi))+x_{k,n}(\xi))-f(k,x_{k,n}(\xi))\}, 
\end{array}
\end{displaymath}
and
\begin{displaymath}
\begin{array}{rcl}
L(n,\nu)&=&\nu + \vartheta(n;(n,\nu)) \\\\
&=&\nu+\hspace{-0.05cm}\sum\limits_{k=-\infty}^{+\infty}G(n,k+1)\{f(k,\vartheta(k;(n,\nu))+y_{k,n}(\nu))-g(k,y_{k,n}(\nu))\}, 
\end{array}
\end{displaymath}
satisfy properties (ii) and (iii) from Definition \ref{TopEq}, which is consequence
of Lemmas \ref{lemme-1} and \ref{lemme-2}. In order to verify property (i), notice that if $n=m$ in the identities 
(\ref{Hom-L})--(\ref{Hom-H}), we obtain that
\begin{displaymath}
L(n,H(n,\xi))=\xi \quad \textnormal{and} \quad H(n,L(n,\nu))=\nu 
\end{displaymath}
for any fixed $n\in \mathbb{Z}$. These identities ensure that $H^{-1}(n,\cdot)=L(n,\cdot)$ for any fixed
$n$. However, the continuity of both maps must be proved. In order to do that, we will follow the approach 
developed by Shi and Xiong \cite{Shi} and Jiang \cite{Jiang} in a continuous framework.
\end{remark}

\section{Proof of Main results}
As stated above, we will prove the continuity properties of the maps $\xi \mapsto H(n,\xi)$
and $\nu \mapsto L(n,\nu)$, for any fixed $n\in \mathbb{Z}$. The proof of Theorem \ref{eq-top-f-g}
follow critically the lines of Jiang \cite[Theorem 2]{Jiang} while the proof of Theorem \ref{eq-top-f-g2} is inspired in Shi and Xiong  \cite[Lemma 10]{Shi}.

\begin{lemma}
\label{gronwall-on}
Let $n\mapsto x(n,k,\xi)$ (resp. $n\mapsto x(n,k,\xi')$) the solution of \textnormal{(\ref{No-lin})} passing through $\xi$ (resp. $\xi'$) 
at $n=k$. Then, it follows that
\begin{equation}
\label{gronwall1}
|x(n,k,\xi)-x(n,k,\xi')|\leq |\xi-\xi'|\exp\Big(\sum\limits_{p=k}^{n-1}(||A_{p}-I||+r_{p})\Big) \quad \textnormal{if} \quad n>k
\end{equation}
and
\begin{equation}
\label{gronwall2}
|x(n,k,\xi)-x(n,k,\xi')|\leq |\xi-\xi'|\exp\Big(\sum\limits_{p=n}^{k-1}(||A_{p}-I||+r_{p})\Big) \quad \textnormal{if} \quad n<k.
\end{equation}
\end{lemma}

\begin{proof}
We will prove only the case $n>k$, the other one can be done similarly. It is straightforward to see that
\begin{displaymath}
x(n,k,\xi)=\xi+\sum\limits_{p=k}^{n-1}(A_{p}-I)x(p,k,\xi)+f(p,x(p,k,\xi)). 
\end{displaymath}

By using \textbf{(A2)}, we have 
\begin{displaymath}
|x(n,k,\xi)-x(n,k,\xi')| \leq |\xi-\xi'|+\sum\limits_{p=k}^{n-1}(||A_{p}-I||+r_{p})|x(p,k,\xi)-x(p,k,\xi')|. 
\end{displaymath}

Finally, by the discrete Gronwall's inequality (see \emph{e.g.}, \cite[Lemma 4.32]{Elaydi}), we have (\ref{gronwall1}).

\end{proof}

\subsection{Proof of Theorem \ref{eq-top-f-g}}
 We will give the proof (in three steps) only for the map $H$ since the other one can be done analogously.

\noindent\emph{Step 1: Preliminary facts.} 
As the identity is a continuous map, we only need to prove that the map $\xi \mapsto \chi(n;(n,\xi))$ is continuous
for any fixed $n$. Now, let us recall that $n\mapsto \chi(n;(m,\xi))$ is the unique bounded solution of (\ref{auxiliary2}), which can be obtained
as the limit of the succesive approximations as done in Lemma \ref{bounded-Q}:
\begin{displaymath}
\chi_{j+1}(n;(m,\xi))=\sum\limits_{k=-\infty}^{+\infty}\hspace{-0.2cm}G(n,k+1)\{g(k,\chi_{j}(k;(m,\xi))+x_{k,m}(\xi))-f(k,x_{k,m}(\xi))\}, 
\end{displaymath}
such that
\begin{displaymath}
\lim\limits_{j\to +\infty}\chi_{j}(n;(m,\xi))=\chi(n;(m,\xi)),
\end{displaymath}
uniformly on $\mathbb{Z}$, which implies that, for any $\varepsilon>0$, there exists $J(\varepsilon)\in \mathbb{N}$ such that
\begin{equation}
\label{CVU}
|\chi(n;(n,\xi))-\chi_{j}(n;(n,\xi))|<\frac{1}{3}\varepsilon \quad \textnormal{for any} \quad j>J.
\end{equation}

On the other hand, by \textbf{(H4)}, we know that for any $\varepsilon>0$, there exists $\ell(\varepsilon)>1$ such that
\begin{equation}
\label{estimacion1}
 \sum\limits_{k=-\infty}^{n-1-\ell}K\exp\Big(-\sum\limits_{p=k+1}^{n}a_{p}\Big)\Delta_{k}+
\sum\limits_{k=n+\ell}^{\infty}K\exp\Big(-\sum\limits_{p=n}^{k+1}a_{p}\Big)\Delta_{k}<\frac{\varepsilon}{2}\big(1-\frac{\theta}{3}\big),
\end{equation}
where $\Delta_{k}$ is defined by
\begin{displaymath}
\begin{array}{rcl}
\Delta_{k}&=&g\big(k,\chi(n;(n,\xi)+x_{k,n}(\xi)\big)-g\big(k,\chi(n;(n,\xi')+x_{k,n}(\xi')\big)\\\\
&&\hspace{-0.25cm}+\hspace{0.2cm}f\big(k,x_{k,n}(\xi')\big)-f\big(k,x_{k,n}(\xi)\big) .
\end{array}
\end{displaymath}

\noindent\emph{Step 2: Claim.} Given $\ell(\varepsilon)\in \mathbb{N}$ defined in (\ref{estimacion1}). For any $j$, there exists $\delta_{j}(\varepsilon,\ell,n)>0$ such that
\begin{equation}
\label{Induccion}
|\chi_{j}(n;(n,\xi))-\chi_{j}(n;(n,\xi'))|<\frac{1}{3}\varepsilon \quad \textnormal{if} \quad |\xi-\xi'|<\delta_{j}.
\end{equation}

\noindent\emph{Step 3: End of proof.} Finally, if $|\xi-\xi'|<\delta_{j}$ with $j>J$, then
\begin{displaymath}
\begin{array}{rcl}
|\chi(n;(n,\xi))-\chi(n;(n,\xi'))|&\leq&|\chi(n;(n,\xi))-\chi_{j}(n;(n,\xi))|\\\\
&&+|\chi_{j}(n;(n,\xi))-\chi_{j}(n;(n,\xi'))|\\\\
&&+|\chi_{j}(n;(n,\xi'))-\chi(n;(n,\xi'))|\\\\
&<&\frac{1}{3}\varepsilon+\frac{1}{3}\varepsilon+\frac{1}{3}\varepsilon=\varepsilon.
\end{array}
\end{displaymath}
and the continuity of $\xi \mapsto \xi+\chi(n;(n,\xi))$ follows.
\medskip

\noindent\emph{Proof of Claim:} The proof will be made by induction by considering an initial term
$$
\chi_{0}(n;(n,\xi))=\chi_{0}(n;(n,\xi'))=\phi \in \ell_{\infty}(\mathbb{Z}) \quad \textnormal{with} \quad |\phi|_{\infty}<B.
$$
and supposing that (\ref{Induccion}) is verified for some $j$ as inductive hypothesis. Now, we have that 
\begin{displaymath}
 \begin{array}{rcl}
\chi_{j+1}(n;(n,\xi))-\chi_{j+1}(n;(n,\xi')) &= & \sum\limits_{k=-\infty}^{\infty}G(n,k+1)\Delta_{k}(g)-\sum\limits_{k=-\infty}^{\infty}G(n,k+1)\Delta_{k}(f)\\\\
&=&\underbrace{\sum\limits_{k=-\infty}^{n-1-\ell}G(n,k+1)[\Delta_{k}(g-f)]
+\sum\limits_{k=n+\ell}^{\infty}G(n,k+1)[\Delta_{k}(g-f)]}_{=A}\\\\
& &+ \underbrace{\sum\limits_{k=n-\ell}^{n-1}G(n,k+1)\Delta_{k}(g)}_{=B_{1}}+\underbrace{\sum\limits_{k=n}^{n+\ell-1}G(n,k+1)\Delta_{k}(g)}_{=B_{2}}\\\\
& &- \underbrace{\sum\limits_{k=n-\ell}^{n-1}G(n,k+1)\Delta_{k}(f)}_{=C_{1}}-\underbrace{\sum\limits_{k=n}^{n+\ell-1}G(n,k+1)\Delta_{k}(f)}_{=C_{2}},\\\\
\end{array}
\end{displaymath}
where $\ell$ is the same as in (\ref{estimacion1}), and $\Delta_{k}(g)$, $\Delta_{k}(f)$ and $\Delta_{k}(g-f)$ are described by:
\begin{itemize}
\item[] $\Delta_{k}(g)=g(k,\chi_{j}(k;(n,\xi))+x_{k,n}(\xi))-g(k,\chi_{j}(k;(n,\xi'))+x_{k,n}(\xi'))$,\\
\item[] $\Delta_{k}(f)=f(k,x_{k,n}(\xi'))-f(k,x_{k,n}(\xi))$,\\
\item[] $\Delta_{k}(g-f)=\Delta_{k}(g)-\Delta_{k}(f)$.
\end{itemize}



By (\ref{estimacion1}), we have that 
\begin{equation}
\label{bound1}
|A|\leq \frac{\varepsilon}{2}\big(1-\frac{\theta}{3}\big).
\end{equation}



In order to estimate $|B|$, by using (\ref{GDD-def}),\textbf{(A2)}, inductive hipothesis and Lemma \ref{gronwall-on}, we can deduce:
\begin{displaymath}
\begin{array}{rcl}
|B_{1}|&\leq &\sum\limits_{k=n-\ell}^{n-1}K\exp\Big(-\sum\limits_{p=k+1}^{n}a_{p}\Big)r_{k}\{|\chi_{j}(k;(n,\xi))-\chi_{j}(k;(n,\xi'))|+|x_{k,n}(\xi)-x_{k,n}(\xi')|\}\\\\
&\leq &\sum\limits_{k=n-\ell}^{n-1}K\exp\Big(-\sum\limits_{p=k+1}^{n}a_{p}\Big)r_{k}\Big\{\frac{1}{3}\varepsilon+|\xi-\xi'|\exp\Big(\sum\limits_{l=k}^{n-1}\{||A_{l}-I||+r_{l}\}\Big)\Big\}\\\\
\end{array}
\end{displaymath}
and
\begin{displaymath}
\begin{array}{rcl}
|B_{2}|&\leq &  \sum\limits_{k=n}^{n+\ell-1} K\exp\Big(-\sum\limits_{p=n}^{k+1}a_{p}\Big)r_{k}\Big\{\frac{1}{3}\varepsilon+|\xi-\xi'|\exp\Big(\sum\limits_{l=n}^{k-1}\{||A_{l}-I||+r_{l}\}\Big)\Big\}.
\end{array}
\end{displaymath}

Analogously, we can verify that
\begin{displaymath}
\begin{array}{rcl}
|C_{1}| &\leq &\sum\limits_{k=n-\ell}^{n-1}K\exp\Big(-\sum\limits_{p=k+1}^{n}a_{p}\Big)r_{k}|x_{k,n}(\xi)-x_{k,n}(\xi')|\\\\
&\leq &|\xi-\xi'|\sum\limits_{k=n-\ell}^{n-1}K\exp\Big(-\sum\limits_{p=k+1}^{n}a_{p}\Big)r_{k}\exp\Big(\sum\limits_{l=k}^{n-1}\{||A_{l}-I||+r_{l}\}\Big).\\\\\
\end{array}
\end{displaymath}
and
\begin{displaymath}
\begin{array}{rcl}
|C_{2}| &\leq &|\xi-\xi'|\sum\limits_{k=n}^{n+\ell-1}K\exp\Big(-\sum\limits_{p=n}^{k+1}a_{p}\Big)r_{k}\exp\Big(\sum\limits_{l=n}^{k-1}\{||A_{l}-I||+r_{l}\}\Big).\\\\\
\end{array}
\end{displaymath}

By using \textbf{(H3)}, we can deduce that
\begin{displaymath}
|B_{1}|+|B_{2}|\leq  \frac{\varepsilon}{3}\theta + |\xi-\xi'|\Gamma(n,\ell) \quad \textnormal{and} \quad
|C_{1}|+|C_{2}|\leq  |\xi-\xi'|\Gamma(n,\ell),
\end{displaymath}
where $\Gamma(n,\ell)$ a finite term is defined by
\begin{displaymath}
\begin{array}{rcl}
\Gamma(n,\ell)&= &\sum\limits_{k=n-\ell}^{n-1}K\exp\Big(-\sum\limits_{p=k+1}^{n}a_{p}\Big)r_{k}\exp\Big(\sum\limits_{l=k}^{n-1}\{||A_{l}-I||+r_{l}\}\Big)\\\\
& & +
\sum\limits_{k=n}^{n+\ell-1}K\exp\Big(-\sum\limits_{p=n}^{k+1}a_{p}\Big)r_{k}\exp\Big(\sum\limits_{l=n}^{k-1}\{||A_{l}-I||+r_{l}\}\Big)
\end{array}
\end{displaymath}

Now, we can deduce that
\begin{displaymath}
 \begin{array}{rcl}
\chi_{j+1}(n;(n,\xi))-\chi_{j+1}(n;(n,\xi')) &= & |A|+|B|+|C|\\\\
&\leq  & \displaystyle \frac{\varepsilon}{2}\big(1-\frac{\theta}{3}\big)+\frac{\varepsilon}{3}\theta+2|\xi-\xi'|\Gamma(n,\ell).
\end{array}
\end{displaymath}

When choosing $\delta_{j+1}=\min\Big\{\delta_{j},\Big(1-\frac{\theta}{3}\Big)\frac{\varepsilon}{4\Gamma(n,\ell)}\Big\}$, we can see that
(\ref{Induccion}) is verified and the claim follows. $\square$

\begin{remark}
A careful examination of the inductive proof of (\ref{Induccion}) show us that $\delta_{j}$
can be dependent of $n$ since we cannot prove that $\Gamma(n,\ell)$ has an upper bound independent of $n$.
This fact has been analized in the continuous framework by Jiang \cite[p.484]{Jiang} but is not clear for us.
\end{remark}

\subsection{Proof of Corollary \ref{cor2}}
We only need to prove that $\Gamma(n,\ell)$ has an upper bound does not depend on $n$. Indeed, by \textbf{(A2)}, \textbf{(H3)}
and (\ref{Stepanov}), we can deduce that
\begin{displaymath}
\begin{array}{rcl}
\Gamma(n,\ell)& \leq &\exp\Big(\sum\limits_{l=n-\ell}^{n+\ell}\{||A_{l}-I||+r_{l}\}\Big)N(n,r)\\\\
              & \leq &\exp\Big(2\{M\ell+M_{0}\ell\}\Big)\theta    
\end{array}
\end{displaymath}
and the result follows. $\square$

\subsection{Proof of Theorem \ref{eq-top-f-g2}}
Firstly, note that the topological equivalence is a direct consequence of Theorem \ref{eq-top-f-g}. Indeed,
\textbf{(H1)} and \textbf{(H3)} are equivalent to \textbf{(D1)} and \textbf{(D2)}. On the other hand, \textbf{(H2)}
is always satisfied since 
\begin{displaymath}
N(n,F+G)\leq K(F+G)\frac{1+e^{-\alpha}}{1-e^{-\alpha}}=B.
\end{displaymath}
Finally, \textbf{(H4)}--\textbf{(H5)}
are a consequence of \textbf{(D1)--(D2)} as stated in Section 2 and all the hypotheses os Theorem \ref{eq-top-f-g}
are satisified, which implies topological equivalence.

Moreover, by following the lines of the proof of Corollary \ref{cor2}, we can deduce that $\Gamma(n,\ell)$ has
an upper bound independent of $n$, and consequently $\delta_{j}$ in (\ref{Induccion}) also. This fact allows to prove the
uniform continuity of  $\xi \mapsto H(t,\xi)$ and  $\nu \mapsto L(t,\nu)$.

Now, we will prove that the map $\xi \mapsto H(n,\xi)$ is H\"older continous for any $n\in \mathbb{Z}$. The other one can be done
in a similar way. As before, we have that
\begin{displaymath}
 \begin{array}{rcl}
|\chi(n;(n,\xi))-\chi(n;(n,\xi'))| & \leq & \sum\limits_{k=-\infty}^{\infty}G(n,k+1)|\Delta_{k}(g)|+\sum\limits_{k=-\infty}^{\infty}G(n,k+1)|\Delta_{k}(f)|\\\\
&\leq &\underbrace{2\sum\limits_{k=-\infty}^{n-1-\ell}G(n,k+1)[F+G]
+2\sum\limits_{k=n+\ell}^{\infty}G(n,k+1)[F+G]}_{=\mathcal{A}}\\\\
& &+ \underbrace{\sum\limits_{k=n-\ell}^{n-1}Ke^{-\alpha(n-k-1)}\Delta_{k}(g)}_{=\mathcal{B}_{1}}+\underbrace{\sum\limits_{k=n}^{n+\ell-1}Ke^{-\alpha(k+1-n)}\Delta_{k}(g)}_{=\mathcal{B}_{2}}\\\\
& &- \underbrace{\sum\limits_{k=n-\ell}^{n-1}Ke^{-\alpha(n-k-1)}\Delta_{k}(f)}_{=\mathcal{C}_{1}}-\underbrace{\sum\limits_{k=n}^{n+\ell-1}Ke^{-\alpha(k+1-n)} \Delta_{k}(f)}_{=\mathcal{C}_{2}}.\\\\
\end{array}
\end{displaymath}

The reader can deduce that
\begin{displaymath}
|\mathcal{A}|\leq \frac{2K(F+G)}{1-e^{-\alpha}}e^{-\ell \alpha}. 
\end{displaymath}

On the other hand, by using \textbf{(A2)} and Lemma \ref{gronwall-on}, we can deduce that
\begin{displaymath}
\begin{array}{rcl}
|\mathcal{B}_{1}|& \leq & \sum\limits_{k=n-\ell}^{n-1}Ke^{-\alpha(n-k-1)}r\Big\{||\chi(\cdot;(n,\xi))-\chi(\cdot;(n,\xi'))||_{\infty}+
|x_{k,n}(\xi)-x_{k,n}(\xi')|\Big\}\\\\
&\leq &   \sum\limits_{k=n-\ell}^{n-1}Ke^{-\alpha(n-k-1)}r\Big\{||\chi(\cdot;(n,\xi))-\chi(\cdot;(n,\xi'))||_{\infty}+
|\xi-\xi'|e^{(M+r)(n-1-k)}\Big\}\\\\
&\leq &   \sum\limits_{k=n-\ell}^{n-1}Ke^{-\alpha(n-k-1)}r\Big\{||\chi(\cdot;(n,\xi))-\chi(\cdot;(n,\xi'))||_{\infty}+
|\xi-\xi'|e^{(M+r)(\ell-1)}\Big\},\\\\
\end{array}
\end{displaymath}
where 
\begin{displaymath}
||\chi(\cdot;(n,\xi))-\chi(\cdot;(n,\xi'))||_{\infty}=\sup\limits_{j\in \mathbb{Z}}|\chi(j;(n,\xi))-\chi(j;(n,\xi'))|.
\end{displaymath}

Similarly, it follows that
\begin{displaymath}
\begin{array}{rcl}
|\mathcal{B}_{2}|
&\leq &   \sum\limits_{k=n}^{n+\ell-1}Ke^{-\alpha(k+1-n)}r\Big\{||\chi(\cdot;(n,\xi))-\chi(\cdot;(n,\xi'))||_{\infty}+
|\xi-\xi'|e^{(M+r)(\ell-2)}\Big\},\\\\
\end{array}
\end{displaymath}
which implies that
\begin{displaymath}
|\mathcal{B}_{1}|+|\mathcal{B}_{2}|\leq  \theta ||\chi(\cdot;(n,\xi))-\chi(\cdot;(n,\xi'))||_{\infty}+\theta|\xi-\xi'|e^{(M+r)\ell},
\end{displaymath}
where 
$$
\theta=\sum\limits_{k=-\infty}^{n-1}Kre^{-\alpha(n-k-1)}+\sum\limits_{k=n}^{\infty}Kre^{-\alpha(k+1-n)}=Kr\Big\{\frac{1+e^{-\alpha}}{1-e^{-\alpha}}\Big\}<1.
$$

The inequality
\begin{displaymath}
|\mathcal{C}_{1}|+|\mathcal{C}_{2}|\leq \theta|\xi-\xi'|e^{(M+r)\ell},
\end{displaymath}
can be deduce as above. Now, it follows that
\begin{displaymath}
|\chi(n;(n,\xi))-\chi(n;(n,\xi'))|\leq \frac{2K(F+G)}{(1-e^{-\alpha})(1-\theta)}e^{-\alpha \ell}+\frac{2\theta}{1-\theta}|\xi-\xi'|e^{(M+r)\ell} 
\end{displaymath}

Let us assume that $|\xi-\xi'|<1$ and let us choose  
$$
\ell =\frac{1}{\alpha}\ln\Big(\frac{1}{|\xi-\xi'|}\Big)
$$
and introduce the constants
$$
D_{1}=1+\frac{2K(F+G)}{(1-e^{-\alpha})(1-\theta)} \quad \textnormal{and} \quad
D_{2}=\frac{2\theta}{1-\theta}.
$$

Finally, a careful computation shows that
\begin{displaymath}
\begin{array}{rcl}
|h(n,\xi)-h(n,\xi')|&\leq & D_{1}|\xi-\xi'|+D_{2}|\xi-\xi'|^{1-(\frac{M+r}{\alpha})} \\\\
&\leq & (D_{1}+D_{2})|\xi-\xi|^{1-(\frac{M+r}{\alpha})}, 
\end{array}
\end{displaymath}
and the result follows. $\square$

\end{document}